\documentclass{amsart}

\usepackage{amsmath}
\usepackage{amsthm,thmtools}
\usepackage{amssymb}
\usepackage{ytableau}
\usepackage{tikz-cd}
\usepackage{hyperref}
\usepackage{enumitem}
\usepackage{stmaryrd}
\usepackage{microtype}
\usepackage{tabularx}
\usepackage{array}
\usepackage[english]{babel}
\usetikzlibrary{babel}

\newcommand{\NN}{\mathbb{N}}

\newcommand{\QQ}{\mathbb{Q}}

\newcommand{\ZZ}{\mathbb{Z}}
\newcommand{\ZZZ}{\widehat{\ZZ}}

\newcommand{\op}{\mathrm{op}}

\renewcommand{\det}{\mathrm{det}}

\let\M\relax
\DeclareMathOperator{\M}{M}

\DeclareMathOperator{\Hom}{Hom}

\DeclareMathOperator{\GL}{GL}

\newcommand{\HOM}{\mathcal{H}\! \mathit{om}}
\newcommand{\END}{\mathcal{E}\! \mathit{nd}}

\newcommand{\sets}{\mathsf{Sets}}
\newcommand{\setswith}[1]{\psh(#1)}

\newcommand{\sh}{\mathsf{Sh}}
\newcommand{\psh}{\mathsf{PSh}}
\newcommand{\pts}{\mathsf{Pts}}

\newcommand{\Flat}{\mathsf{Flat}}

\newcommand{\fp}{\mathfrak{p}}
\newcommand{\fq}{\mathfrak{q}}

\newcommand{\ns}{\mathrm{ns}}

\let\min\relax

\newcommand{\min}{\mathrm{min}}

\newtheorem{definition}{Definition}[subsection]
\newtheorem{proposition}[definition]{Proposition}

\newtheorem{theorem}[definition]{Theorem}
\newtheorem{corollary}[definition]{Corollary}
\newtheorem{lemma}[definition]{Lemma}
\newtheorem{example}[definition]{Example}
\newtheorem{remark}[definition]{Remark}

\tikzset{
b/.style={bend left=10},
bb/.style={bend left},
cl/.style={outer sep=-1pt},
}

\newcommand\rurl[1]{%
  \href{http://#1}{\nolinkurl{#1}}%
}
\newcommand\rsurl[1]{%
  \href{https://#1}{\nolinkurl{#1}}%
}

\title{Localization of monoids and topos theory}
\author{Jens Hemelaer}
\address{Department of Mathematics, University of Antwerp \\ 
 Middelheimlaan 1, B-2020 Antwerp (Belgium) \\ {\tt jens.hemelaer@uantwerpen.be}}

\AtBeginDocument{%
   \def\MR#1{}
}

\begin{document}
\begin{abstract}
Let $M$ be a monoid that is embeddable in a group. We consider the topos $\psh(M)$ of sets equipped with a right $M$-action, and we study the subtoposes that are of monoid type, i.e.\ the subtoposes that are again of the form $\psh(N)$ for $N$ a monoid. Our main result is that every subtopos of monoid type can be obtained by localization at a prime ideal of $M$. Conversely, we show that localization at a prime ideal produces a subtopos if and only if $M$ has the right Ore property with respect to the complement of the prime ideal. We demonstrate our calculations in some examples: free monoids, two monoids related to the Connes--Consani Arithmetic Site, and torus knot monoids.
\end{abstract}
\maketitle
\setcounter{tocdepth}{1}
\tableofcontents

\section{Introduction}
\label{sec:introduction}

For $M$ a monoid, the category of sets equipped with a right $M$-action is a Grothendieck topos. More precisely it is the presheaf topos $\psh(M)$, where $M$ is seen as a category with one object. Studying the category $\psh(M)$ provides insight into the structure of the monoid $M$ itself, see e.g.\ \cite{MAC}, and in particular, topos-theoretic properties of $\psh(M)$ can be translated to algebraic properties of $M$ \cite{hr1}.

In this paper, we will discuss the \emph{subtoposes of monoid type} of $\psh(M)$, by which we mean subtoposes of the form $\psh(N) \subseteq \psh(M)$, for $N$ another monoid. We focus on (not necessarily commutative) monoids $M$ that admit an embedding into a group $G$. Note that this is a strictly stronger condition than $M$ being cancellative \cite{bush}.

In \cite[\S 3.3]{pirashvili}, it is shown that for a commutative monoid $M$ and a multiplicative subset $S \subseteq M$, there is a geometric embedding
\begin{equation*}
\psh(S^{-1}M) \to \psh(M),
\end{equation*}
with $S^{-1}M$ the monoid obtained from $M$ by adding a formal inverse to each element in $S$. As pointed out in \cite[Lemma 1.1]{cortinas-haesemeyer-walker-weibel}, we can assume, without changing $S^{-1}M$, that $S=M-\fp$ for $\fp \subseteq M$ a prime ideal. We then write $M_\fp = S^{-1}M$, mimicking the notation from ring theory.

In this article, we will prove the following, for monoids $M$ that admit an embedding to a group. As we will discuss in Section \ref{sec:right-ore}, we still get a geometric embedding
\begin{equation*}
\psh(M_\fp) \subseteq \psh(M),
\end{equation*}
for a prime ideal $\fp \subseteq M$, but only under the condition that $M$ satisfies the right Ore property with respect to $M-\fp$. The main result of the paper is then the converse statement, that any subtopos of monoid type of $\psh(M)$ is of the form $\psh(M_\fp) \subseteq \psh(M)$ for $\fp$ a prime ideal, with $M$ right Ore with respect to $M-\fp$, see Theorem \ref{thm:complete-classification}.

The reason we restrict to monoids embeddable in a group, is that in this case we can use the results from \cite{topological-groupoid}. 
There it was shown that for $M \subseteq G$ a submonoid of a group $G$, there is an equivalence of categories
\begin{equation*}
\psh(M) \simeq \sh_G(X)
\end{equation*}
for some topological space $X$ on which the discrete group $G$ has a continuous left group action. The topological space $X$ has an explicit construction, and as remarked in \cite{topological-groupoid} this gives a method of explicitly computing the category of points of $\psh(M)$. In Section \ref{sec:points}, we will recall this construction, and then give explicit descriptions of the associated flat left $M$-sets and their endomorphism monoids.

In Section \ref{sec:examples}, we will discuss our framework in detail in four examples: the free monoid on $k$ generators, the monoid $\mathbb{N}^{\times}_+$ of nonzero natural numbers under multiplication, the monoid $\M_2^\ns(\ZZ)$ of $2\times 2$ integer matrices with nonzero determinant, and the torus knot monoid $\langle{a,b:a^k=b^l}\rangle$ for $k,l \geq 2$. 

For the second and third example, the calculation of the points is already discussed extensively in the literature, so these examples are only included to illustrate our framework, and to prepare for Section \ref{sec:return-to-examples}. The monoid $M =\mathbb{N}^{\times}_+$ underlies the Connes--Consani Arithmetic Site, and for a calculation of the topos-theoretic points we refer to \cite{connes-consani}, \cite{connes-consani-geometry-as} and \cite{llb-covers-general-version}. The case $M=\M_2^\ns(\ZZ)$ is a noncommutative variation on the Arithmetic Site, here we refer to \cite{arithmtop} \cite{thesis-jens}. For more examples related to the Arithmetic Site, see also \cite{sagnier} \cite{llb-three}. For finitely generated commutative monoids, the topos-theoretic points are classified up to isomorphism by the set of prime ideals \cite[Theorem 5.2.10]{pirashvili}, while for finite monoids, points can be described in terms of idempotents \cite[\S 3]{pirashvili-finite}.

In Section \ref{sec:subtoposes}, we will use the calculations from Section \ref{sec:points} to arrive at our main result, which is that the subtoposes of monoid type of $\psh(M)$ are all of the form $\psh(M_\fp) \subseteq \psh(M)$, for $\fp$ a prime ideal of $M$ such that $M$ is right Ore with respect to $M-\fp$. In the commutative case, the right Ore condition is automatically satisfied, so here we find that the subtoposes of monoid type correspond precisely to the prime ideals.

Finally, in Section \ref{sec:return-to-examples} we will revisit our earlier four examples, to demonstrate how our main result can be applied in practice.

\section{Prime ideals and the right Ore condition} \label{sec:right-ore}

\subsection{Prime ideals}

Recall that a left (resp.\ right) ideal of $M$ is a subset of $M$ which is closed under multiplication on the left by $M$ (resp.\ on the right). A two-sided ideal is a subset of $M$ that is both a left ideal and a right ideal.

\begin{definition}
A \emph{prime ideal} of $M$ is a two-sided ideal $\fp \subseteq M$ such that $1 \notin \fp$ and for all $a,b \in M$, if $ab \in \fp$ then either $a \in \fp$ or $b \in \fp$.
\end{definition}

So a two-sided ideal is a prime ideal if and only if its complement is a submonoid. The empty subset $\varnothing \subseteq M$ is always a prime ideal. If $\phi: M \to N$ is a monoid homomorphism, then whenever $\fp$ is a prime ideal of $N$, we have that $\phi^{-1}(\fp)$ is a prime ideal of $M$.

To compute the prime ideals of a monoid, the following result from \cite{pirashvili-spectrum} is very useful.

\begin{proposition}[{\cite[Lemma 2.1]{pirashvili-spectrum}}] \label{prop:prime-pirashvili}
Let $M$ be a monoid. Then there is a bijective correspondence between prime ideals of $M$ and monoid homomorphisms $M \to \{0,1\}$, where $\{0,1\}$ is equipped with the usual multiplication law.
The monoid homomorphism $\phi: M \to \{0,1\}$ corresponds to the prime ideal $\phi^{-1}(0) \subseteq M$.
\end{proposition}

In \cite{pirashvili-spectrum}, attention is restricted to commutative monoids, but the same proof works for noncommutative monoids.

In particular, a finitely generated monoid will have only finitely prime ideals, see \cite[Lemma 1.5]{cortinas-haesemeyer-walker-weibel}.

\subsection{Localization}

For a subset $S \subseteq M$, recall that we can construct a \emph{localization} $S^{-1}M$ of $M$ by formally inverting the elements of $S$. Starting from a presentation of $M$ by generators and relations, we add for each $s \in S$ a new generator $s^{-1}$ and relations $ss^{-1} = s^{-1}s = 1$. Now $S^{-1}M$ has the universal property that for any monoid homomorphism $\phi : M \to N$ such that $\phi(s)$ is invertible (on both sides) for all $s \in S$, there is a unique factorization of $\phi$ through the natural map $M \to S^{-1}M$. This universal property uniquely determines $S^{-1}M$; in particular, our construction in terms of generators and relations does not depend on the presentation of $M$ that we start with.

As a special case, the localization of $M$ at a prime ideal $\fp$ is defined as $M_\fp = S^{-1}M$ for $S = M-\fp$.

\begin{lemma} \label{lmm:S-becomes-invertible}
Let $M$ be a monoid and $\fp \subseteq M$ a prime ideal. Then there is an equality $M-\fp = M_\fp^\times \cap M$.
\end{lemma}
\begin{proof}
If $m \in M-\fp$, then by definition $m$ becomes invertible in $M_\fp$. 

Conversely, suppose that $m$ becomes invertible in $M_\fp$.
Take the monoid homomorphism $\phi : M \to \{0,1\}$ corresponding to $\fp$, as in Proposition \ref{prop:prime-pirashvili}.
By the universal property of $M_\fp$, we can extend $\phi$ to a monoid homomorphism $\psi : M_\fp \to \{0,1\}$.
If $m$ becomes invertible in $M_\fp$, then $\phi(m)=\psi(m)=1$, so $m \in M-\fp$.
\end{proof}

\subsection{Tensor products and flatness}

Let $M$ be an arbitrary monoid. For a right $M$-set $X$ and a left $M$-set $A$, we define the tensor product as the quotient
\begin{equation*}
X \otimes_M A = \{ (x,a) : x \in X,~ a \in A \}/\!\sim
\end{equation*}
by an equivalence relation $\sim$ is generated by the relations
\begin{equation*}
(xm,a) \sim (x,ma)
\end{equation*}
for $m \in M$. The equivalence class of $(x,a)$ is written as $x \otimes a$. 

Tensoring with $A$ defines a colimit-preserving functor
\begin{gather*}
-\otimes_M A : \psh(M) \to \sets \\
X \mapsto X \otimes_M A.
\end{gather*}
If this functor preserves finite limits, then we say that $A$ is \emph{flat}.

The left $M$-set $A$ can equivalently be seen as a functor $M \to \sets$, where $M$ is seen as a one object category, and the functor is defined by sending the unique object to the set $A$ and each $m \in M$ to the function $A \to A,~ a \mapsto ma$. Under this correspondence, $A$ is flat as left $M$-set if and only if it is flat as a functor (see e.g.\ \cite[Proposition 1.8]{hr1}). Concretely, this means that $A$ is flat as a left $M$-set if and only if
\begin{enumerate}
\item[(F1)] $A \neq \varnothing$;
\item[(F2)] for all elements $a,b \in A$ there is an element $c \in A$ and elements $m,n \in M$ such that $a = mc$ and $b = nc$;
\item[(F3)] if $ma = na$ for $a \in A$ and $m,n \in A$, then there is an element $b \in A$ and an element $s \in M$ such that $a = sb$ and $ms = ns$. 
\end{enumerate}

\begin{example}
Let $\NN_+^\times$ be the monoid of nonzero natural numbers under multiplication, and let $\QQ_+^\times = \{q \in \QQ : q > 0 \}$. Then the left $\mathbb{N}^{\times}_+$-action on $\QQ_+^\times$, given by multiplication, is flat. The second criterion above is satisfied because we can bring $a$ and $b$ to a common denominator, and the third criterion is satisfied because of cancellativity. 
\end{example}

\begin{example} \label{eg:flatness-free-group}
Let $M$ be the free monoid in two variables $x$ and $y$, and let $G$ be the free group in the same variables $x$ and $y$. Then $G$ has a left $M$-action by multiplication, but this action is not flat. More precisely, in the second criterion above, take $a = x^{-1}$ and $b = y^{-1}$. If there is an element $c \in G$ and elements $m,n \in M$ with $x^{-1} = mc$, $y^{-1} = nc$, then $xm = c^{-1} = yn$. This gives a contradiction, because $xM \cap yM = \varnothing$. So $G$ is not flat as left $M$-set.
\end{example}

In the first example above, $\QQ_+^\times$ is the groupification of $\mathbb{N}^{\times}_+$, and in the second example $G$ is the groupification of $M$. Note that groupification coincides with localization at the empty prime ideal.
In the next subsection, we will discuss a criterion to decide whether $M_\fp$ is flat as a left $M$-set, for $\fp \subseteq M$ an arbitrary prime ideal.

\subsection{Right Ore condition}
The following concept is well-known in ring theory, see e.g.\ \cite[\S 9.1]{cohn}.

\begin{definition} \label{def:right-ore}
Let $M$ be a monoid and $S \subseteq M$ a multiplicative subset. Then we say that $M$ is \emph{right Ore with respect to $S$} if for any $m \in M$ and $s \in S$ there are $t \in S$ and $n \in M$ such that $mt = sn$.
\end{definition}

A monoid $M$ is called \emph{right Ore} if it is right Ore with respect to itself \cite[Chapter I, Definition 3.18]{MAC}.

\begin{proposition} \label{prop:flat-iff-right-ore}
Let $M \subseteq G$ be a submonoid of a group $G$, and let $\fp \subseteq M$ be a prime ideal. Then $M_\fp$ is flat as left $M$-set if and only if $M$ is right Ore with respect to $S = M-\fp$.
\end{proposition}
\begin{proof}
From Lemma \ref{lmm:S-becomes-invertible}, we know that $S = M_\fp^\times \cap M$.

First suppose that $M_\fp$ is flat as left $M$-set. Take $m \in M$ and $s \in S$. Then using flatness of $M_\fp$, we can find $m_1, m_2 \in M$ and $a \in M_\fp$ such that
\begin{gather*}
\begin{split}
s^{-1}m = m_1 a \\
1 = m_2 a.
\end{split}
\end{gather*}
It then follows that $mm_2 = sm_1$, with $m_2 \in M_\fp^\times \cap M = S$. So $M$ is right Ore with respect to $S$.

Conversely, suppose that $M$ is right Ore with respect to $S$. Then for every $m \in M$ and $s \in S$ we can find $n \in M$ and $t \in S$ such that $mt = sn$, so $s^{-1}m = nt^{-1}$. We can use this property to show that any element of $M_\fp$ can be written as $ms^{-1}$ for some $m \in M$ and $s \in S$. Now let $ms^{-1}$ and $nt^{-1}$ be two elements of $M_\fp$, with $m,n\in M$ and $s,t \in S$. To prove that $M_\fp$ is flat as a left $M$-set, it is enough to find $r \in S$ and $m_1,m_2 \in M$ such that $ms^{-1} = mm_1r^{-1}$ and $nt^{-1} = nm_2r^{-1}$. First, apply the right Ore property to find $m_1\in M$, $m_2 \in S$ such that $sm_1 = tm_2$. Now take $r = tm_2$. Because both $t$ and $m_2$ are in $S$, so is $r$. We then get $s^{-1} = m_1r^{-1}$ and $t^{-1} = m_2r^{-1}$, which is what we needed.
\end{proof}

\begin{remark}
The arguments in the above proof are well-known in ring theory. See for example \cite[Lemma II.3.1(ii)]{artin-notes}.
\end{remark}

\begin{example}
Let $G$ be the groupification of $M$. Then we can write $G=M_\fp$ for $\fp=\varnothing$. It then follows from Proposition \ref{prop:flat-iff-right-ore} that $G$ is flat if and only if $M$ is right Ore.
\end{example}

From Proposition \ref{prop:flat-iff-right-ore}, we can deduce the following:

\begin{theorem} \label{thm:prime-right-ore-then-subtopos}
Let $M \subseteq G$ be a submonoid of a group $G$, and let $\fp \subseteq M$ be a prime ideal.
Then there is a subtopos of the form $\psh(M_\fp) \subseteq \psh(M)$, with inverse image functor $-\otimes_M M_\fp$,
if and only if $M$ is right Ore with respect to $S=M-\fp$.
\end{theorem}
\begin{proof}
If there is such a subtopos, then $M_\fp$ must be flat as a left $M$-set, 
because the inverse image functor of a geometric morphism preserves finite limits.
By Proposition \ref{prop:flat-iff-right-ore}, it then follows that $M$ is right Ore with respect to $S$.

Conversely, if $M$ is right Ore with respect to $S$, then $M_\fp$ is flat as a left $M$-set by Proposition \ref{prop:flat-iff-right-ore}, so we get a geometric morphism $\psh(M_\fp) \to \psh(M)$ with inverse image functor $-\otimes_M M_\fp$. The direct image functor is the right adjoint of $-\otimes_M M_\fp$, which is the forgetful functor restricting the $M_\fp$-action to an $M$-action.
\end{proof}

In the remainder of the paper, we will show that every subtopos 
of monoid type of $\psh(M)$ is of the above form,
still assuming that $M$ is a submonoid of a group.

\section{\texorpdfstring{Flat left $M$-sets as topos points}{Flat left M-sets as topos points}}
\label{sec:points}

\subsection{The category of points}

For a Grothendieck topos $\mathcal{E}$, a \emph{point} is a geometric morphism $p : \sets \to \mathcal{E}$, or equivalently, a functor $p^* : \mathcal{E} \to \sets$ that preserves colimits and finite limits. The points of a topos form a category, with as morphisms $p \to q$ the natural transformations $p^* \to q^*$.

It follows from Diaconescu's theorem that:

\begin{proposition}[{See e.g.\ \cite[Corollary 1.9]{hr1}}]
The category of points of $\psh(M)$ is equivalent to the category of flat left $M$-sets and homomorphisms of left $M$-sets between them. 
\end{proposition}

Now suppose that $M$ is a submonoid of a group $G$. It was shown in \cite{topological-groupoid} that there is an equivalence of categories
\begin{equation*}
\psh(M) \simeq \sh_G(X)
\end{equation*}
for some topological space $X$ on which the discrete group $G$ has a continuous left group action. The topological space $X$ has an explicit construction, and as remarked in \cite{topological-groupoid} this gives a method of explicitly computing the category of points of $\psh(M)$. Because, $\psh(M) \simeq \sh_G(X)$, the two toposes have the same categories of points. Further, the category of points of $\sh_G(X)$ has the following description:

\begin{proposition} \label{prop:points-equivariant-sheaves}
Let $X$ be a topological space and let $G$ be a discrete group with a continuous left action on $X$. Then the category $\pts_G(X)$ of points of $\sh_G(X)$ has
\begin{itemize}
\item as objects the elements of the sobrification $\widehat{X}$ of $X$;
\item as morphisms $x \to y$ the elements $g \in G$ such that $x \leq gy$;
\item composition defined via multiplication in the opposite group $G^\op$.
\end{itemize}
Here $\leq$ denotes the specialization order on $X$, so $x \leq gy$ means that $x$ lies in the closure of $gy$, or equivalently that $x \in U \Rightarrow gy \in U$ for every open set $U \subseteq X$. The left $G$-action on $\widehat{X}$ is the one induced by the $G$-action on $X$.
\end{proposition}

For a proof, see \cite[Subsection 1.3.5]{thesis-jens}.

We will recall from \cite{topological-groupoid} how to construct the topological space $X$, its sobrification $\widehat{X}$, and the action of $G$. Afterwards, we will show how to compute the flat left $M$-set $A_x$ corresponding to an element $x \in \widehat{X}$. Similarly, we will be able to describe the morphisms of left $M$-sets $A_x \to A_y$ in terms of the group action of $G$.

Let $\leq$ be the partial order on $G/M^\times$ given by
\begin{equation}
[a] \leq [b] ~\Leftrightarrow~ \exists m \in M,~ am = b.
\end{equation} 
Let $X$ be the topological space with as elements the elements of $G/M^\times$, and as open sets the upwards closed sets for the above partial order. There is a natural left $G$-action on $X$ by multiplication on the left, and this action is continuous.

With the topological space $X$ and the $G$-action as above, there is an equivalence $\psh(M) \simeq \sh_G(X)$, see \cite{topological-groupoid}. To get a complete description of the category of points according to \ref{prop:points-equivariant-sheaves}, the most difficult step that remains is the computation of the sobrification $\hat{X}$ of $X$.

\subsection{Computing the sobrification $\widehat{X}$.} One construction of the sobrification $\hat{X}$ is as follows. The elements of $\hat{X}$ are the irreducible closed subsets of $X$. The function $X \to \hat{X}$ then sends an element $x \in X$ to the closure of $x$ (this is always an irreducible closed subset). The irreducible closed subsets of $X$ are ordered under inclusion, and this induces a partial order on $\hat{X}$, extending the specialization order on $X$. The open sets for $\hat{X}$ are then the sets of the form
\begin{equation*}
\hat{U} = \{ x \in \hat{X} : \exists m \in U,~ m \leq x \}
\end{equation*}
for an open subset $U \subseteq X$. This induces a bijection between the frame of open subsets of $X$ and the frame of open subsets of $\hat{X}$.

\begin{remark}
In our case, recall that the topology on $X$ comes from a partial order $\leq$ on $X$, by taking as open sets the subsets $U \subseteq X$ that are upwards closed. This partial order is then also the specialization order for the topology, i.e.\ $x \leq y$ if and only if $x$ is in the closure of $\{y\}$.

Some terminology: the topology on the poset $X$ is usually called the \emph{Alexandrov topology}, and $X$ with this topology is called \emph{Alexandrov-discrete}. The induced topology on $\hat{X}$ is then called the \emph{Scott topology}.
\end{remark}

How do we characterize the irreducible closed subsets of $X = G/M^\times$? The open sets of $X$ are the upwards closed sets, so the closed sets are the downwards closed sets. We claim that the irreducible closed subsets correspond to the ideals in $X$. This appears as an exercise in \cite{compendium-2}, but we give a proof below for completeness.

\begin{lemma}[{\cite[Exercise V-4.9]{compendium-2}}] \label{lmm:ideals}
Let $X$ be a poset with the Alexandrov topology. The irreducible closed subsets of $X$ are precisely the ideals of $X$, i.e.\ the subsets $F \subseteq X$ such that
\begin{enumerate}
\item $F$ is non-empty;
\item $F$ is downwards closed, i.e.\ $x \leq y$ and $y \in F$ implies $x \in F$;
\item $F$ is upwards directed, i.e.\ if $x, y \in F$ then there is a $z \in F$ with $x,y \leq z$.
\end{enumerate}
\end{lemma}
\begin{proof}
We already saw that $F \subseteq X$ is closed if and only if it is downwards closed. Further, $F$ is irreducible if $F = F' \cup F''$ with $F',F'' \subseteq X$ closed implies that $F = F'$ or $F=F''$. Equivalently, for any two open subsets $U,V \subseteq X$ with $U\cap F \neq \varnothing$ and $V \cap F\neq\varnothing$ we have that $U \cap V \cap F \neq \varnothing$. So if $F$ is irreducible, then it is upwards directed, by taking $U$ and $V$ to be the upwards closures of $x$ and $y$, respectively. Conversely, suppose that $F$ is upwards directed. Take $U,V \subseteq X$ open (i.e.\ upwards directed) such that $U \cap F \neq \varnothing$ and $V \cap F \neq \varnothing$. Take $x \in U \cap F$ and $y \in V \cap F$, and further take $z \in F$ with $x,y \leq z$. Then $z \in U \cap V \cap F$. 
\end{proof}

The $G$-action on $\hat{X}$ is now defined as $g \cdot F = \{ gx : x \in F \}$ for an element $g \in G$ and a ideal $F \subseteq X$. 

\subsection{Translation from points to flat left $M$-sets}

Because $\psh(M)\simeq \sh_G(X)$, there is an equivalence of categories
\begin{equation} \label{eq:category-of-points}
\Flat(M) \simeq \mathbf{Pts}_G(X)
\end{equation}
where $\Flat(M)$ denotes the category of flat left $M$-sets and $\mathbf{Pts}_G(X)$ denotes the category of points of $\sh_G(X)$, as described by Proposition \ref{prop:points-equivariant-sheaves}.

Because $X$ is a $T_0$-space, the natural map to the sobrification $X \to \hat{X}$ is an injection. So we will identify $X$ with a subspace of $\hat{X}$. The elements of $X$ all lie in the same $G$-orbit, so by Proposition \ref{prop:points-equivariant-sheaves} they define isomorphic points of $\sh_G(X)$. From the construction of the equivalence $\sh_G(X) \simeq \setswith{M}$ in \cite{topological-groupoid} it follows that the point $[1] \in X$ corresponds to the flat left $M$-set $M$. We can now determine the flat left $M$-set $A_x$ corresponding to an arbitrary point $x \in \hat{X}$, because the elements of $A_x$ are given by morphisms $M \to A_x$, which are by the equivalence (\ref{eq:category-of-points}) the same as the elements $g \in G$ such that $1 \leq gx$. So we find that:
\begin{equation*}
A_x = \{ g \in G : gx \geq 1 \}.
\end{equation*}
The left $M$-action on $A_x$ is given by multiplication. Further, for an element $h \in G$, we find that
\begin{equation} \label{eq:base-change-Ay}
A_{hx} = \{ g \in G : ghx \geq 1 \} = A_x \, h^{-1}.
\end{equation}

\subsection{Morphisms of flat left $M$-sets} \label{ssec:points-morphisms} For two points $x,y \in \hat{X}$, we can now also compute the set of homomorphisms $A_x \to A_y$. From the equivalence (\ref{eq:category-of-points}) and Proposition \ref{prop:points-equivariant-sheaves}, we can make the identification:

\begin{equation*}
\HOM^M\!(A_x,A_y) = \{ g \in G : x \leq gy \}
\end{equation*}
(we write the $\HOM^M$ for morphisms of left $M$-sets and $\HOM_M$ for morphisms of right $M$-sets). Here the morphism corresponding to $g \in G$ is the one that sends $a \in A_x$ to $ag \in A_y$. Note that composition of morphisms corresponds to multiplication in the opposite group $G^\op$. In particular, we find that
\begin{equation*}
\END^M\!(A_x) = \{ g \in G : gx \geq x  \}^\op \subseteq G^\op.
\end{equation*}
is a submonoid of $G^\op$.

We will write $M_x = \END^M\!(A_x)^\op$, so
\begin{equation*}
M_x = \{ g \in G : gx \geq x \} \subseteq G.
\end{equation*}
In particular, $M_{[1]} = M$ as submonoids of $G$. The inclusion $M_x \to \END^M(A_x)^\op$ defines a right $M_x$-action on $A_x$, and this right $M_x$-action is compatible with the left $M$-action. It is given by $a \cdot m = am$, for $a \in A_x$ and $m \in M_x$, where multiplication on the right hand side of the equation happens in $G$. 

For an element $h \in G$, we find that
\begin{equation} \label{eq:base-change-My}
M_{hx} = hM_x h^{-1}.
\end{equation}

\subsection{Ideals containing $[1]$.} If $F\subseteq X$ is an ideal, and $[g] \in F$ is an element with representative $g \in G$, then $g^{-1}\cdot F$ is an ideal containing $[1]$. So up to isomorphism, every point of $\sh_G(X)$ is given by an ideal $F \subseteq X$ that contains $[1]$.

Filters containing $[1]$ are easier in the following sense:

\begin{lemma} \label{lmm:ideals-on-Y}
Let $F$ be an ideal containing $[1]$. Let
\begin{equation*}
Y = \{ x \in X : x \geq [1] \} \subseteq X.
\end{equation*}
Then $F \cap Y \subseteq Y$ is an ideal of $Y$. Conversely, if $F' \subseteq Y$ is an ideal of $Y$, and let $F$ be the downwards closure of $F'$ in $X$. Then $F \subseteq X$ is an ideal of $X$. These two procedures are inverse to each other, so $\hat{Y} = \{ x \in \hat{X} : x \geq [1] \}$.
\end{lemma}
\begin{proof}
This follows directly from the axioms of an ideal, as stated in Lemma \ref{lmm:ideals}.
\end{proof}

By the above, we only have to determine the ideals on $Y=M/M^\times$ in order to understand the ideals on $X$. Indeed, each ideal on $X$ is of the form $g\cdot F$, for $g \in G$ and $F$ the downwards closure of an ideal on $Y$.

\section{Examples}
\label{sec:examples}

In Section \ref{sec:subtoposes}, we will use the calculations from the Section \ref{sec:points} to study the subtoposes of monoid type of $\psh(M)$. But first, we show how to apply the results of Section \ref{sec:points} to determine the category of flat left $M$-sets in some examples. We use the notations of Section \ref{sec:points}.

\subsection{Free monoids} \label{ssec:examples-free}

Let $M = \langle{a_1,\dots,a_k}\rangle$ be the free monoid on $k$ variables $a_1,\dots,a_k$. To determine the points of $\psh(M)$, we first have to determine the ideals on $Y = M/M^\times$, where $Y$ is ordered by the relation $[x] \leq [y] \Leftrightarrow \exists m \in M,~ y = xm$.

In this case, the group of units $M^\times$ is trivial, so $Y = M$. 

The finite ideals in $Y$ are of the form $y = \{1,x_0,x_0x_1,x_0x_1x_2,\dots,x_0x_1\dots x_r\}$ with $x_i \in \{a_1,\dots,a_k\}$ for all $0 \leq i \leq r$. Using the formulas \eqref{eq:base-change-Ay} \eqref{eq:base-change-My} we find here $A_y = Mm^{-1}$ and $M_y = mMm^{-1}$. 

The infinite ideals on the poset $Y=M$ are more interesting. They are given by the infinite words
\begin{equation*}
x_0x_1x_2\dots
\end{equation*}
with $x_i \in \{a_1,\dots,a_k\}$ for all $i \in \NN$. The corresponding ideal of $Y$ is the subset
\begin{equation*}
\{ 1, x_0, x_0x_1, x_0x_1x_2, \dots \} \subseteq Y
\end{equation*}
and by taking the downwards closure we find the corresponding ideal in $X$.

Note that two distinct infinite ideals of $Y$ are necessarily incomparable. This means that if $y \in \hat{Y}$ is not finite, then $y \leq gy$ if and only if $y = gy$. It follows that all endomorphism monoids of flat left $M$-sets are groups. We compute some examples (we write $a = a_1$ and $b=a_2$).
\begin{itemize}
\item Take the point $y = aaa\dots$. Then $y = gy$ if and only if $g \in \langle{a,a^{-1}}\rangle$. So the endomorphism monoid is isomorphic to $(\ZZ,+)$.
\item Take the point $y = abbaabbaabba\dots$. Then $y = gy$ if and only if $g$ is in the free subgroup generated by the element $abba$. So again, the endomorphism monoid is isomorphic to $(\ZZ,+)$. 
\end{itemize}
In all cases where the point is represented by a ``periodic'' sequence that repeats the exact same pattern infinitely many times, we will be in a similar situation, where the endomorphism monoid is isomorphic to $(\ZZ,+)$. 
More generally, if the point is of the form $gy$ with $y$ a periodic sequence and $g \in G$, then we can use the formula $M_{gy} = gM_y g^{-1}$. So again, the endomorphism monoid is isomorphic to $(\ZZ,+)$. The conclusion is that the endomorphism monoid of the flat left $M$-set associated to $y$ is isomorphic to $(\ZZ,+)$ whenever the sequence $y$ is \emph{eventually periodic}.

We claim that if $y$ is not eventually periodic, then the endomorphism monoid of the associated flat left $M$-set is trivial. Take $g \in G$. Write $y = x_0x_1x_2\dots$ and $gy = x_0'x_1'x_2'\dots$. If $gy \geq 1$ then necessarily $g=x_0'x_1'\dots x_r'x_s^{-1}x_{s-1}^{-1}\dots x_0^{-1}$ for some $r,s \in \NN$. We can then rewrite $gy$ as
\begin{equation*}
gy = x_0'\dots x_r' x_{s+1}x_{s+2}\dots 
\end{equation*}
In other words, $x_{s+i}=x_{r+i}'$ for all $i \geq 1$. Now suppose that $gy=y$. Then with $r$ and $s$ as above, we have that $x_{s+i}=x_{r+i}$ for all $i \geq 1$. If $r\neq s$ then this implies that $y$ is eventually periodic, with period dividing $|r-s|$. If $r=s$, then it follows from $g=x_0'x_1'\dots x_r'x_s^{-1}x_{s-1}^{-1}\dots x_0^{-1}$ that $g$ is trivial. We conclude that, if $y$ is not eventually periodic, then the endomorphism monoid of the associated flat left $M$-set is trivial.

\subsection{The Arithmetic Site}

The Connes--Consani Arithmetic Site \cite{connes-consani} \cite{connes-consani-geometry-as} has as underlying topos the topos $\psh(M)$ for $M = \mathbb{N}^{\times}_+$ the monoid of nonzero natural numbers under multiplication. In this case, again $M^\times = \{1\}$, so $Y = \mathbb{N}^{\times}_+$, with as partial ordering the division relation. The groupification of $\mathbb{N}^{\times}_+$ is the group of strictly positive rational numbers $\QQ_+^\times$.

In \cite[Proposition 3]{llb-covers-general-version}, it is shown that ideals of $Y$ correspond to \emph{supernatural numbers}, i.e.\ formal infinite products $\prod_{p} p^{e_p}$ indexed by the prime numbers, with $e_p \in \NN \cup \{+\infty\}$ for each prime $p$. The ideal $F_s$ corresponding to a supernatural number $s$ is
\begin{equation*}
F_s = \{ n \in \mathbb{N}^{\times}_+ : n | s \},
\end{equation*}
see the proof of \cite[Proposition 3]{llb-covers-general-version}. Here we take the natural division relation on supernatural numbers.

The flat left $\mathbb{N}^{\times}_+$-set associated to the supernatural number $y = \prod_p p^{e_p}$ is 
\begin{equation} \label{eq:Ay-arithm-site}
A_y = \{ \tfrac{a}{b} \in \QQ_+^* : b | y  \},
\end{equation}
as shown in the proof of \cite[Theorem 2]{llb-covers-general-version}. We can also compute that
\begin{equation} \label{eq:My-arithm-site}
M_y = \{ \tfrac{a}{b} \in \QQ_+^* : \forall p \text{ prime},~p | b \Rightarrow p \notin \Sigma_y \}
\end{equation}
with $\Sigma_y$ the set of primes $p$ such that the exponent $e_p$ is finite.

If we instead take $M = \ZZ^\ns$ the nonzero integers under multiplication, then $Y = M/M^*$ can still be identified with $\mathbb{N}^{\times}_+$ with partial ordering given by division. So the filters on $Y$ are the same as above, and the analogues of \eqref{eq:Ay-arithm-site} and \eqref{eq:My-arithm-site} are
\begin{equation*}
A_y = \{ \tfrac{a}{b} \in \QQ^* : b | y  \},\quad M_y = \{ \tfrac{a}{b} \in \QQ^* : \forall p \text{ prime},~p | b \Rightarrow p \notin \Sigma_y \}.
\end{equation*}

\subsection{Matrices with integer coefficients}

Let $M = \M_2^\ns(\ZZ)$ be the $2\times 2$-matrices with integer coefficients and nonzero determinant. The groupification is $G = \GL_2(\QQ)$. We already know from \cite{arithmtop} that the poset of filters on $Y = M/M^\times$ is given by
\begin{equation*}
\M_2(\ZZZ)\slash \GL_2(\ZZZ)
\end{equation*}
with $[x] \geq [y]$ if and only if there is some $m \in \M_2(\ZZZ)$ such that $x = ym$. 
For $y \in \M_2(\ZZZ)\slash \GL_2(\ZZZ)$, the associated flat $M$-set is
\begin{equation*}
A_y = \{ g \in \GL_2(\QQ) : gy \in \M_2(\ZZZ) \}.
\end{equation*}
Further, we have
\begin{equation*}
M_y = \{ g \in \GL_2(\QQ) : gy \geq y \}.
\end{equation*}
A more detailed overview of what $A_y$ and $M_y$ can look like, would lead us too far.
We can however discuss some special cases:
\begin{itemize}
\item If $y$ is the zero matrix, then $A_y = M_y = \GL_2(\QQ)$. 
\item Suppose that $y = \begin{pmatrix}
1 & 0 \\
0 & 0 
\end{pmatrix}$. We write an element $g \in G$ as $g = \begin{pmatrix}
a & b \\
c & d
\end{pmatrix}$, so $gy = \begin{pmatrix}
a & 0 \\
c & 0 
\end{pmatrix}$. Now $gy \geq y$ if there is an $m \in \M_2(\widehat{\ZZ})$ such that $gy = ym$. If we let $m = \begin{pmatrix}
z_{11} & z_{12} \\
z_{21} & z_{22}
\end{pmatrix}$ then $gy = ym$ holds if and only if $c = 0$ and $z_{11}=a$, $z_{12} = 0$. It follows that $gy \geq y$ if and only if $c = 0$ and $a \in \ZZ$. We can symbolically write this down as
\begin{equation*}
M_y = \begin{pmatrix}
\ZZ & \QQ  \\
0 & \QQ
\end{pmatrix} \cap \GL_2(\QQ).
\end{equation*}
\end{itemize}

Note that the category of points of $\psh(\M_2^\ns(\ZZ))$ is equivalent to the category of torsionfree abelian groups of rank $2$ and injective homomorphisms between them, see \cite[Proposition 4.1]{thesis-jens}.

\subsection{Torus knot monoids} Consider the monoid $M = \langle{a,b : a^k = b^l}\rangle$ with $k,l \geq 2$ natural numbers. We will call $M$ the \emph{torus knot monoid}, following \cite{elder-kalka}. 

What are the ideals on $Y = M/M^\times$? Note that the unit group $M^\times$ is trivial, and that the ordering relation on $Y = M$ is then $x \leq y \Leftrightarrow \exists m \in M,~ y = mx$. To simplify the discussion, we will write $M$ as a submonoid of a more well-behaved monoid, as follows.

The quotient $H = \langle{a,b~:~ a^k=b^l=1}\rangle$ of $M$ is a group. We can further introduce a degree function on $M$, namely the monoid homomorphism $\deg : M \to \NN$ sending $a$ to $l$ and $b$ to $k$, with $\NN$ seen as monoid under addition. We claim that the map $M \to \NN \times H,~ m \mapsto (\deg(m),[m])$ is injective. First, we write
\begin{equation*}
\delta(h) = \min_{[m]=h} \deg(m).
\end{equation*}
where the minimum is taken over all $m \in M$ with $[m]=h$. Now take $m,m' \in M$ such that $[m]=[m']$ and $\deg(m)=\deg(m')$. We prove that $m=m'$, using induction on $\deg(m)=\deg(m')$. As a base case, consider the situation $\deg(m)=\deg(m')=\delta([m])$. From $[m]=[m']$ it then follows that either $m=m'$ or $m,m' \in McM$ for $c = a^k=b^l$. In the second case, we can write $m=m_0c$, $m'=m_0'c$ for $m_0,m_0' \in M$. This gives a contradiction, because $\deg(m_0) = \deg(m)-\deg(c) < \deg(m)$ and $[m_0] = [m]=g$. In the induction step, we get from $[m]=[m']$ again that either $m=m'$ or otherwise $m=m_0c$, $m'=m_0'c$ for $m_0,m_0' \in M$. In the second case, we get $m_0=m_0'$ by induction, from which $m=m'$.

Via the above procedure, we now view $M$ as a submonoid of $\NN \times H$. The elements of $M$ are then of the form $(\delta(h)+\lambda kl, h)$ for a unique $\lambda \in \NN$, $h \in H$. We call $\lambda$ the \emph{level} of $m$, and $\nu = \delta(h)$ the \emph{niveau} of $m$, in the spirit of \cite[\S2.1]{arithmtop}. We write the level and niveau as $\lambda(m)$ resp.\ $\nu(m)$ if we want to stress the dependency on $m$.

Recall our notation $c = a^k=b^l$.

\begin{lemma} \label{lmm:level}
For $m \in M$ an element of level $\lambda$, we have $m \in Mc^\lambda$ and $\lambda$ is maximal with this property.
\end{lemma}
\begin{proof}
Let $(\delta(h)+\lambda kl, h)$ be an arbitrary element of $M$. Then we can write
\begin{equation*}
(\delta(h)+\lambda kl, h) = (\delta(h), h)(\lambda kl, 1)
\end{equation*}
with $(\lambda kl, h) = c^\lambda$ and $(\delta(h),h) \notin Mc$.
\end{proof}

We now distinguish three types of ideals of the poset $Y=M$. 

If within an ideal $F \subseteq Y$ the level is unbounded, i.e.\ $\sup_{m \in F} \lambda(m) = +\infty$, then using Lemma \ref{lmm:level}, we see that $F$ contains $c^N$ for each $N \in \NN$, because $F$ is downwards closed.
Now let $m = (\delta(h) + \mu kl, h)$ be an arbitrary element of $M$. Then
\begin{equation*}
(\delta(h^{-1}),h^{-1})(\delta(h)+\mu kl, h) = (\delta(h^{-1}+\delta(h)+\mu kl, 1),
\end{equation*}
and the element on the right hand set is a power of $c$. Because $F$ is downwards closed, it follows that $m \in F$. But $m\in M$ was arbitrary, so $F=M$. Conversely, we can check that $F=M$ is indeed an ideal. If $y$ is the point corresponding to the ideal $F=M$, then $A_y = M_y = G$, with $G$ the groupification of $M$.

Now suppose that within $F$, $\lambda$ attains a maximum $\lambda = \lambda_0$. Because $F$ is downwards closed, we find using Lemma \ref{lmm:level} that $c^{\lambda_0} \in F$. Because $F$ is upwards directed, it then follows that $F$ is the downwards closure of the elements $m \in F$ with $\lambda(m) = \lambda_0$. This means we can write $F = c^{\lambda_0} \cdot F_0$ for a filter $F_0 \subseteq Y$ containing only elements of level $0$. To simplify the discussion, we restrict our attention to the latter type of filters, so $F=F_0$. For each element $m \in F_0$, there is a unique word $w$ in the letters $a$ and $b$ such that the class of $w$ in $M$ coincides with $m$ (because the equivalence $a^k\sim b^l$ can only be applied if $w$ contains either $a^k$ or $b^l$ as a substring, but then $\lambda(m)\geq 1$).
Just as in Subsection \ref{ssec:examples-free}, we conclude that either
\begin{equation} \label{eq:infinite-word}
F = \{1, x_0, x_0x_1, x_0x_1x_2,\dots \}
\end{equation}
with $x_0x_1x_2\dots$ an infinite word in $\{a,b\}$ that does not contain $a^k$ or $b^l$ as a substring, or 
\begin{equation} \label{eq:finite-word}
F = \{1, x_0,x_0x_1, \dots, x_0x_1\dots x_r \}
\end{equation}
for a finite word $x_0x_1\dots x_r$ in $\{a,b\}$ that, again, does not contain $a^k$ or $b^l$ as a substring. This corresponds to the situation where within $F$ the niveau $\nu$ is unbounded resp.\ bounded.
If $y$ is the point corresponding to $F$, for $F$ as in \eqref{eq:infinite-word}, then we find
\begin{equation*}
A_y = \{ mx_r^{-1}\dots x_0^{-1} : m \in M,~ r \in \NN \}.
\end{equation*}
The computation of $M_y$ is analogous to that in Subsection \ref{ssec:examples-free}: we find that $M_y \cong \ZZ$ whenever $y = x_0x_1x_2\dots$ is an eventually periodic word, and $M_y = \{1\}$ otherwise. If $F$ corresponds to a finite word as in \eqref{eq:finite-word}, then we can write $m=x_0x_1\dots x_r$ and then we can use the formulas \eqref{eq:base-change-Ay} and \eqref{eq:base-change-My} to find $A_y = Mm^{-1}$ and $M_y = mMm^{-1}$. 

Reminder: this is only for the case where $\lambda_0 = 0$; for general $\lambda_0$ we find $F=c^{\lambda_0}F_0$, for $F_0$ as in \eqref{eq:infinite-word} or \eqref{eq:finite-word}. If $y$ and $y_0$ are the points corresponding to $F$ resp.\ $F_0$, then we can use \eqref{eq:base-change-Ay} and \eqref{eq:base-change-My} to conclude $A_y = A_{y_0}c^{-\lambda_0}$ and $M_y = M_{y_0}$.

\section{Subtoposes of monoid type}
\label{sec:subtoposes}

\subsection{A first necessary condition}

Let $M$ be a monoid embedded in a group $M \subseteq G$. Let $X$ be the space constructed earlier, such that
\begin{equation*}
\sh_G(X) \simeq \setswith{M}.
\end{equation*}
We want to determine the subtoposes \emph{of monoid type} of $\setswith{M}$, i.e.\ those subtoposes that are themselves of the form $\setswith{N}$ for some monoid $N$. Like all presheaf toposes, these subtoposes have enough points. So we can use the following:
\begin{proposition} \label{prop:subtopos-equivariant}
Let $X$ be a sober topological space with a left continuous action of a discrete group $G$. Then the subtoposes of $\sh_G(X)$ that have enough points are in bijective correspondence with $G$-invariant sober subspaces $Z \subseteq X$. The subtopos corresponding to $Z \subseteq X$ is $\sh_G(Z)$.
\end{proposition}
For a proof, see \cite[Proposition 3.27]{thesis-jens}.

When can we write $\sh_G(Z)$ as $\setswith{N}$ for some monoid $N$? If $\sh_G(Z) \simeq \setswith{N}$, then in particular $\pts_G(Z) \simeq \Flat(N)$ by looking at the categories of points. We will identify each flat left $N$-set with its corresponding point in $Z$. Let $x \in Z$ be the point corresponding to the flat left $N$-set $N$. By computing its monoid of endomorphisms in two ways, we find that $N^\op = \END^N(N) = \END^M(A_x) = M_x^\op$. 

After taking opposites, we get
\begin{equation*}
N = M_x = \{ g \in G : gx \geq x \},
\end{equation*}
in particular, $N$ is a submonoid of $G$.

Let $i : \psh(N) \to \psh(M)$ corresponding to the inclusion $\sh_G(Z) \to \sh_G(X)$. The geometric morphism $i$ has a inverse image functor $i^* : \psh(M) \to \psh(N)$ that has a right adjoint $i_*$. The functor $i^*$ is of the form
\begin{equation*}
i^*(X) \simeq X \otimes_M P
\end{equation*}
for some set $P$ with compatible left $M$-action and right $N$-action. The direct image functor is the right adjoint
\begin{equation*}
i_*(Y) \simeq \HOM_N(P,Y),
\end{equation*}
see e.g.\ \cite[Proposition 1.5]{hr1}. Because $i^*$ preserves finite limits, $P$ is flat as left $M$-set. 

The geometric morphism $i$ induces a functor between categories of points
\begin{equation*}
\begin{split}
\Flat(N) \longrightarrow \Flat(M) \\
B \mapsto P \otimes_N B.
\end{split}
\end{equation*}
We already assumed above that this functor sends $N$ to $A_x$. So we can make the identification $P = A_x$. The right $N$-action on $A_x$ then agrees with the right $M_x$-action on $A_x$ as in Subsection \ref{ssec:points-morphisms}, via the earlier identification $N=M_x$.

We have now shown the following:

\begin{lemma} \label{lmm:subtoposes}
Let $M$ be a submonoid of a group $G$, and let $i : \psh(N) \to \psh(M)$ be a geometric embedding, for some monoid $N$. Then, with the notations of Section \ref{sec:points} and up to equivalence, the geometric embedding can be written in the form
\begin{equation} \label{eq:form-of-subtopos}
\begin{tikzcd}
\psh(M_y) \ar[r,bend right,"{\HOM_{M_y}(A_y,-)}"'] & \psh(M) \ar[l,bend right,"{-\otimes_M A_y}"']
\end{tikzcd}
\end{equation}
for some $y \in \hat{X}$.
\end{lemma}

\begin{remark} \label{rem:equivalent-subtoposes}
For $y \in \hat{X}$ and $h \in G$, suppose that we have two geometric morphisms $f$ and $g$ of the following form:
\begin{equation*}
\begin{tikzcd}
\psh(M_y) \ar[r,bend right,"{\HOM_{M_y}(A_y,-)}"'] & \psh(M) \ar[l,bend right,"{-\otimes_M A_y}"']
\end{tikzcd}
\end{equation*}
\begin{equation*}
\begin{tikzcd}
\psh(M_{hy}) \ar[r,bend right,"{\HOM_{M_{hy}}(A_{hy},-)}"'] & \psh(M) \ar[l,bend right,"{-\otimes_M A_{hy}}"']
\end{tikzcd}
\end{equation*}
The function $M_y \to M_{hy},~ m \mapsto hmh^{-1}$ is an isomorphism of monoids. If we consider $A_{hy}$ as right $M_{hy}$-set, and then look at the induced right $M$-action via $a \cdot m = a \phi(m)$, then the bijection $A_y \to A_{hy},~ a \mapsto ah^{-1}$ is equivariant for both the left $M$-action and the right $M_y$-action.

From the above discussion, we find a diagram
\begin{equation*}
\begin{tikzcd}[row sep=small]
\psh(M_y) \ar[dd,"{U}"'] \ar[near start,rrd,"{f}"] & \\
& & \psh(M)  \\
\psh(M_{hy}) \ar[near start,rru,"{g}"']  & 
\end{tikzcd}
\end{equation*}
commuting up to natural isomorphism, with $U$ an equivalence. So if $-\otimes_M A_y$ defines a subtopos of the form (\ref{eq:form-of-subtopos}), then $- \otimes_M A_{hy}$ defines the same subtopos.
\end{remark}

\begin{example}
Let $M$ be the free monoid on $k$ variables. We know from Subsection \ref{ssec:examples-free} that there are three possibilities for $M_y$:
\begin{itemize}
\item if $y = a_0a_1\dots a_r$ is a finite word, then $M_y \cong M$;
\item if $y = a_0a_1a_2\dots$ is eventually periodic, then $M_y \cong \ZZ$;
\item if $y = a_0a_1a_2\dots$ is not eventually periodic, then $M_y \cong \{1\}$. 
\end{itemize}
From Lemma \ref{lmm:subtoposes}, it now follows that if $\psh(N) \to \psh(M)$ is a geometric embedding for some monoid $N$, then $N$ is isomorphic to $M$, $\ZZ$ or the trivial monoid. We will later exclude the latter two possibilities, meaning that $\psh(M)$ does not have any subtoposes of monoid type other than $\psh(M) \subseteq \psh(M)$ itself.
\end{example}

\subsection{Taking it a step further}

Now consider the right $M$-set $G$ (under multiplication). We show that it is preserved under inverse image.

\begin{lemma} \label{lmm:G-tensor-A_y}
With the notations of Section \ref{sec:points}, consider a geometric embedding of the form
\begin{equation*}
\begin{tikzcd}
\psh(M_y) \ar[r,bend right,"{\HOM_{M_y}(A_y,-)}"'] & \psh(M) \ar[l,bend right,"{-\otimes_M A_y}"']
\end{tikzcd}
\end{equation*}
for some $y \in \hat{X}$. Then the map $G \otimes_M A_y \to G,~ g\otimes a \mapsto ga$ is bijective.
\end{lemma}
\begin{proof}
Because we are working with a geometric embedding, the counit
\begin{equation} \label{eq:G-embedding}
\begin{split}
\HOM_{M_y}(A_y,G) \otimes_M A_y \longrightarrow G \\
\phi \otimes a \mapsto \phi(a)
\end{split}
\end{equation}
is an isomorphism. Now consider the embedding $G \hookrightarrow \HOM_{M_y}(A_y,G)$ sending $g$ to the homomorphism $a \mapsto ga$. Tensoring with $A_y$ preserves monomorphisms, so the restriction
\begin{equation*}
\begin{split}
G \otimes_M A_y \longrightarrow G \\
g \otimes a \mapsto ga
\end{split}
\end{equation*}
of \eqref{eq:G-embedding} is still injective. Because $1 \otimes a$ is mapped to $a$, we see that the map is surjective as well. 
\end{proof}

\begin{lemma} \label{lmm:My-Ay}
With the notations of Section \ref{sec:points}, consider a geometric embedding of the form
\begin{equation*}
\begin{tikzcd}
\psh(M_y) \ar[r,bend right,"{\HOM_{M_y}(A_y,-)}"'] & \psh(M) \ar[l,bend right,"{-\otimes_M A_y}"']
\end{tikzcd}
\end{equation*}
for some $y \in \hat{X}$. Then there is a $g \in G$ such that $M_y = gA_y$.
\end{lemma}
\begin{proof}
The inclusion $M_y \subseteq G$ induces an inclusion $\HOM_{M_y}(A_y,M_y) \subseteq \HOM_{M_y}(A_y,G)$. We can also view $G$ as a subset of $\HOM_{M_y}(A_y,G)$ by sending $g$ to the map $a \mapsto ga$. 
Because we are working with a geometric embedding, we can make the identification $\HOM_{M_y}(A_y,G)\otimes_{M} A_y = G$. Then we also get $\HOM_{M_y}(A_y,M_y) \otimes_M A_y = M_y$. From Lemma \ref{lmm:G-tensor-A_y}, we find $G \otimes_M A_y = G$.

We now define $S$ as the intersection $S = \HOM_{M_y}(A_y,M_y) \cap G$. Tensoring with $A_y$ preserves intersections, so $S \otimes_M A_y = M_y$.

After writing $S = \bigcup_{i \in I} g_i M$ as right $M$-sets and tensoring with $A_y$, we find
\begin{equation*}
M_y = \bigcup_{i \in I} g_iA_y.
\end{equation*}
Take $i_0 \in I$ such that $1 \in g_{i_0}A_y$. Because $1$ generates $M_y$ as right $M_y$-set, we see that $M_y = g_{i_0}A_y$, which is what we wanted to prove.
\end{proof}

\begin{theorem} \label{thm:subtoposes-of-monoid-type}
Let $M$ be a submonoid of a group $G$. Suppose there is an element $y \in \hat{X}$ such that $M_y = A_y$.  Then $y \in \hat{Y}$ and there is a subtopos of monoid type $i : \psh(M_y) \to \psh(M)$ with inverse image functor given by $i^*(X) \simeq X \otimes_M M_y$. Conversely, every subtopos of monoid type of $\psh(M)$ is of this form for a unique $y \in \hat{Y}$ such that $M_y=A_y$. 
\end{theorem}
\begin{proof}
Suppose that $M_y = A_y$. Then $M \subseteq M_y=A_y$, because $M_y=A_y$ contains $1$ and is closed under the left $M$-action by multiplication. From $M \subseteq A_y$ it follows that $y \geq [1]$.

Moreover, since $M_y$ is flat as a left $M$-set, the functor $i^*$ with $i^*(X) = X \otimes_M M_y$ preserves colimits and finite limits, so it is the inverse image functor of a geometric morphism $i : \psh(M_y) \to \psh(M)$. We claim that $i$ is a geometric embedding. First, note that the direct image functor $i_*$ is given by $i_*(Y) \simeq \Hom_{M_y}(M_y,Y) \simeq Y$. More precisely, the direct image functor $i_*$ is the forgetful functor that sends a right $M_y$-set to the same set with the induced right $M$-action. In particular, $i_*$ preserves colimits. To prove that $i$ is a geometric embedding, we want to prove that the counit $\epsilon : i^*i_* \to 1$ is a natural isomorphism. Because $i^*i_*$ preserves colimits, it is enough to show that $\epsilon$ is an isomorphism at the generator $M_y$, or in other words that 
\begin{equation*}
\begin{tikzcd}[row sep=tiny]
M_y \otimes_M M_y \ar[r,"{\epsilon_{M_y}}"] & M_y \\
g \otimes h \ar[r,mapsto] & gh
\end{tikzcd}
\end{equation*}
is an isomorphism. Surjectivity is clear. To prove injectivity, consider the inclusion $M_y \subseteq G$. Applying $i^*$ gives an injective map $j : M_y \otimes_M M_y \to G \otimes_M M_y \simeq G$ where the latter isomorphism follows from Lemma \ref{lmm:G-tensor-A_y}. The injective map $j$ factors as $\epsilon_{M_y}$ followed by the inclusion $M_y \subseteq G$, so $\epsilon_{M_y}$ is injective as well. It follows that $i$ is a geometric embedding.

Conversely, suppose that a subtopos of monoid type of $\psh(M)$ is given. By Lemma \ref{lmm:subtoposes} this subtopos is of the form $i : \psh(M_y) \to \psh(M)$ with $i^*(X) \simeq -\otimes_M A_y$, for some $y \in \hat{X}$. By Lemma \ref{lmm:My-Ay}, we can moreover find some $g \in G$ such that $M_y = gA_y$. We then get $M_{g^{-1}y} = g^{-1}M_y g = A_yg = A_{g^{-1}y}$ using the base change formulas \eqref{eq:base-change-Ay} and \eqref{eq:base-change-My}. By replacing $y$ with $g^{-1}y$, see Remark \ref{rem:equivalent-subtoposes}, we can assume $M_y = A_y$. 

In each $G$-orbit, there can be only one $y \in \hat{X}$ such that $M_y = A_y$. Indeed, suppose that $M_y = A_y$ and $M_{gy} = A_{gy}$. Using equations (\ref{eq:base-change-Ay}) and (\ref{eq:base-change-My}) the latter equation reduces to $gM_y = A_y$. Because $gM_y = M_y$ we get that $g \in M_y$ and $g^{-1} \in M_y$, or in other words $gy \geq y$ and $g^{-1}y \geq y$. The second inequality is equivalent to $y \geq gy$, so we get $gy = y$.
\end{proof}

\begin{corollary}
Let $M$ be a submonoid of a group $G$. Then there is a bijective correspondence between subtoposes of $\psh(M)$ of monoid type (up to equivalence) and monoids $N$ with $M \subseteq N \subseteq G$ that are flat as left $M$-set.
\end{corollary}
\begin{proof}
By Theorem \ref{thm:subtoposes-of-monoid-type} the subtoposes of monoid type are of the form $\psh(M_y) \to \psh(M)$ where $y \in \hat{Y}$ is an element such that $A_y = M_y$. Because $y \in \hat{Y}$, we find that $M_y=A_y$ contains $M$. So by taking $N = A_y = M_y$ we see that each subtopos gives a monoid $N$ as in the statement. 

Conversely, suppose that $N$ is a monoid with $M \subseteq N \subseteq G$ such that $N$ is flat as left $N$-set. Then $N = A_y$ for some $y \in \hat{X}$. Because $M \subseteq N$, we find $y \in \hat{Y}$. We claim that $M_y = A_y = N$. It follows from $y \geq 1$ that $M_y \subseteq A_y$. We can write $M_y$ as $M_y = \{ g \in G : \forall a \in A_y,~ ag \in A_y \}$. Since $A_y = N$ is multiplicatively closed, we see that we also have the converse inclusion $A_y \subseteq M_y$. We conclude that $N = M_y = A_y$. 
\end{proof}

\begin{example} \label{eg:subtoposes-free-monoid}
Let $M = \langle{a_1,\dots,a_k}\rangle$ be the free monoid on $k$ variables. From Subsection \ref{ssec:examples-free} we know that $M_y$ is isomorphic to $M$, $\ZZ$ or the trivial monoid. If $M_y = A_y$, then we have $M \subseteq M_y$ which is only possible if $M_y \cong M$. So $y$ is a finite word, and we see that the equality $M_y = A_y$ is satisfied if and only if $y = 1$. This corresponds to the full subtopos $\psh(M) \subseteq \psh(M)$, and there are no other subtoposes of monoid type.
\end{example}

\subsection{Connection to prime ideals}

Take $M$, $G$ and $X$ as before. We know that any subtopos of monoid type can be written in the form
\begin{equation*}
\begin{tikzcd}
\psh(M_y) \ar[r,bend right,"{\HOM_{M_y}(M_y,-)}"'] & \psh(M) \ar[l,bend right,"{-\otimes_M M_y}"']
\end{tikzcd}
\end{equation*}
for a certain $y \in \hat{X}$ such that $M_y = A_y$. We claim that here $M_y$ is the localization of $M$ at some prime ideal $\fp$.

For $N$ a submonoid of a group $G$, any left invertible element of $N$ is also right invertible, and vice versa. The subset of invertible elements is $$N^\times = \{ n \in N : n^{-1} \in N \},$$ with the inverse taken in $G$. We can then check directly that $N-N^\times$ is a prime ideal of $N$.

Taking $N=M_y$, we find the prime ideal $M_y - M_y^\times$ of $M_y$. We now define $\fp_y$ to be the inverse image of this prime ideal along the inclusion $M \subseteq M_y$, i.e.\ 
\begin{equation*}
\fp_y = (M_y-M_y^\times)\cap M.
\end{equation*}

\begin{proposition} \label{prop:My-is-Ay-then-prime}
Let $y \in \hat{X}$ be an element such that $M_y = A_y$. Then $M_y = M_{\fp_y}$.
\end{proposition}
\begin{proof}
By construction, every element of $S = M-\fp_{y}$ becomes invertible in $M_y$. By the universal property of localization, this means that the inclusion $M \subseteq M_y$ factorizes through $M_{\fp_y}$. Up to isomorphism, we can assume that the factorization takes place via inclusions $M \subseteq M_{\fp_y} \subseteq M_y$. 

It remains to show the converse inclusion $M_y \subseteq M_{\fp_y}$. Take $g \in M_y$. Because $M_y=A_y$, we know that $M_y$ is flat as left $M$-set. So we can take $a \in M_y$ and elements $m,n \in M$ such that $ma = g$ and $na = 1$. We then find $g = mn^{-1}$ with $n \in S$, so $g \in M_{\fp_y}$.
\end{proof}

By combining Theorem \ref{thm:prime-right-ore-then-subtopos}, Theorem \ref{thm:subtoposes-of-monoid-type} and Proposition \ref{prop:My-is-Ay-then-prime}, we can conclude:

\begin{theorem} \label{thm:complete-classification}
Let $M$ be a submonoid of a group. Then the subtoposes of monoid type of $\psh(M)$ are precisely the subtoposes $\psh(M_\fp) \subseteq \psh(M)$, with inverse image functor $-\otimes_M M_\fp$, for $\fp \subseteq M$ a prime ideal such that $M$ is right Ore with respect to $M-\fp$.
\end{theorem}

\section{Returning to our examples}
\label{sec:return-to-examples}

\subsection{Free monoids} 

We already saw in Example \ref{eg:subtoposes-free-monoid} that $\psh(M)$ does not have any nontrivial subtoposes of monoid type, for $M = \langle{a_1,\dots,a_k}\rangle$ the free monoid on $k\geq 2$ variables. 
Using Theorem \ref{thm:complete-classification}, we can view the same phenomenon from a different angle. 

Let $S \subseteq M$ be any multiplicative subset. If $s \in S$ is an element with $s \neq 1$, then we can find an $m \in M$ such that $sM \cap mM = \varnothing$. For example, if $s \in a_1M$, then we can take $m \in a_2M$. It follows that $M$ is right Ore with respect to $S$ if and only if $S = \{1\}$. By Theorem \ref{thm:complete-classification}, the only subtopos of monoid type is then the one corresponding to the prime ideal $\fp = M -\{1\}$, i.e.\ the full subtopos $\psh(M) \subseteq \psh(M)$.

In order to reach this conclusion, it is not necessary to calculate the prime ideals of $M$. However, it might be helpful to take a look at the prime ideals anyway. For $M$ the free monoid on $k$ generators, we deduce from Proposition \ref{prop:prime-pirashvili} that there are $2^k$ prime ideals. Indeed, to construct a monoid homomorphism $M \to \{0,1\}$ we can freely choose either $0$ or $1$ as image for every generator of $M$. The prime ideals of $M$ can then be written as
\begin{equation*}
\fp_J ~=~ \bigcup_{j \in J} Ma_j M
\end{equation*}
for $J \subseteq \{1,\dots,k\}$ any subset.

\subsection{The Arithmetic Site}

We can use Proposition \ref{prop:prime-pirashvili} to compute the prime ideals of $M = \mathbb{N}^{\times}_+$. Because $M$ is a free commutative monoid on countably many generators, we see that the prime ideals are exactly the ideals of the form
\begin{equation} \label{eq:prime-ideal-arithm-site}
\fp = \bigcup_{p \in \Sigma} p\mathbb{N}^{\times}_+
\end{equation}
for $\Sigma$ any set of prime numbers. Because $M$ is commutative, it has the right Ore property with respect to any subset. By Theorem \ref{thm:complete-classification}, we then find that the subtoposes of monoid type for $\psh(\mathbb{N}^{\times}_+)$ are given by
\begin{equation*}
\psh(\mathbb{N}^{\times}_{+,\fp}) \subseteq \psh(\mathbb{N}^{\times}_+)
\end{equation*}
for $\fp$ as in \eqref{eq:prime-ideal-arithm-site}.

For $M = \ZZ^\ns$ the nonzero integers under multiplication, the situation is completely analogous. This time, the prime ideals are of the form
\begin{equation} \label{eq:prime-ideals-Zns}
\fp = \bigcup_{p \in \Sigma} p\ZZ^\ns
\end{equation}
for $\Sigma$ any set of primes. For each prime ideal, there is a subtopos of monoid type $\psh(\ZZ^\ns_\fp) \subseteq \psh(\ZZ^\ns)$, and these are the only subtoposes of monoid type.

\subsection{Matrices with integer coefficients} Consider now the monoid $M = \M^\ns(\ZZ)$. For $\det : M \to \ZZ^\ns$ the determinant map, we find a prime ideal $\det^{-1}(\fp) \subseteq M$ for every prime ideal $\fp$ of $\ZZ^\ns$. For a prime ideal $\fp \subseteq \ZZ^\ns$ as in \eqref{eq:prime-ideals-Zns}, the corresponding prime ideal of $M$ is
\begin{equation*}
\fq = \det^{-1}(\fp) = \{ a \in M : \exists p \in \Sigma,~ p\, | \, \det(a) \}.
\end{equation*}
Let $S$ be the complement of $\fq = \det^{-1}(\fp)$. In other words $S$ consists of the matrices whose determinant is not divisible by any prime in $\Sigma$.

For every $a \in M$, we have $aa^* = \det(a)$, where $a^*$ is the adjugate matrix of $a$, and $\det(a)$ is interpreted as a scalar matrix. From this observation, it follows that adding an inverse to each element of $S$  amounts to the same as adding an inverse to each $n \in \ZZ^\ns - \fp$, with $n$ interpreted as scalar matrix. In other words, we find that $M_{\fq} = M \otimes_{\ZZ^\ns} \ZZ^\ns_{\fp}$, where the tensor product is taken with respect to the inclusion $\ZZ^\ns \subseteq M$, sending an integer to the associated scalar matrix. Because $\ZZ^\ns_{\fp}$ is flat as left $\ZZ^\ns$-set, we know that $M_{\fq}$ will also be flat as left $M$-set. Indeed, $- \otimes_M M_\fq$ is the composition of $- \otimes_M M : \psh(M) \to \psh(\ZZ^\ns)$ followed by $- \otimes_{\ZZ^\ns} \ZZ^\ns_\fp : \psh(\ZZ^\ns) \to \sets$, and a composition of two functors that preserve finite limits will again preserve finite limits. So if a prime ideal of $M$ is of the form $\fq = \det^{-1}(\fp)$, then we get a subtopos $\psh(M_\fq) \subseteq \psh(M)$.

We now claim that every prime ideal of $M$ is of the form $\fq = \det^{-1}(\fp)$. We can again use Proposition \ref{prop:prime-pirashvili}. Every matrix $a \in M$ can be written in its Smith normal form $a = udv$ with $u,v \in M^\times$ and $d$ a diagonal matrix, with each diagonal entry dividing the next one. It follows that $M$ is generated by its units and by the matrices of the form
\begin{equation*}
a_p = \begin{pmatrix}
1 & 0 \\
0 & p
\end{pmatrix}
\end{equation*}
for every prime number $p$. By proposition \ref{prop:prime-pirashvili}, this means that for every set of prime numbers $\Sigma$, there is at most one prime ideal $\fq$ of $M$ such that
\begin{equation*}
a_p \in \fq ~\Leftrightarrow~ p \in \Sigma,
\end{equation*}
namely
\begin{equation}
\fq = \bigcup_{p \in \Sigma} M a_p M.
\end{equation}
This prime ideal is precisely $\fq = \det^{-1}(\fp)$ for $\fp = \bigcup_{p \in \Sigma} p\ZZ^\ns$.

We conclude that $\fp \mapsto \det^{-1}(\fp)$ gives a bijective correspondence between the prime ideals of $\ZZ^\ns$ and those of $M$. This in turn gives us a bijective correspondence between the subtoposes of monoid type for $\psh(\ZZ^\ns)$ and the subtoposes of monoid type for $\psh(M)$, sending $\psh(\ZZ^\ns_\fp) \subseteq \psh(\ZZ)$ to $\psh(M_\fq) \subseteq \psh(M)$, with $\fq = \det^{-1}(\fp)$.

\subsection{Torus knot monoids} Let us return to the torus knot monoid $$M = \langle{a,b : a^k = b^l}\rangle,$$ with $k,l \geq 1$ natural numbers. First, let us determine the prime ideals of $M$. Using Proposition \ref{prop:prime-pirashvili}, the prime ideals correspond to the monoid homomorphisms $\phi : M \to \{0,1\}$. If $\phi(a) = 0$, then $\phi(b)^l = \phi(b^l) = \phi(a^k) = \phi(a)^k = 0$, but then $\phi(b)=0$ as well. Conversely, if $\phi(b) = 0$, then $\phi(a)=0$ as well. So we can either send both $a$ and $b$ to $0$, or both $a$ and $b$ to $1$. The corresponding prime ideals of $M$ are $M - \{1\}$ and $\varnothing$. 

Trivially, $M$ is right Ore with respect to $\{1\}$, and we can also check that $M$ is right Ore with respect to itself. Using Theorem \ref{thm:complete-classification} we conclude that there are precisely two subtoposes of monoid type: localizing at $M-\{1\}$ gives the full subtopos $\psh(M)\subseteq\psh(M)$, while localizing at $\varnothing$ gives $\psh(G) \subseteq \psh(M)$ for $G$ the groupification of $M$ (i.e.\ the corresponding torus knot group).

\section*{Acknowledgements}

The author is a postdoctoral fellow of the Research Foundation Flanders (file number 1276521N).

\newcommand{\etalchar}[1]{$^{#1}$}
\providecommand{\bysame}{\leavevmode\hbox to3em{\hrulefill}\thinspace}
\providecommand{\MR}{\relax\ifhmode\unskip\space\fi MR }
\providecommand{\MRhref}[2]{%
  \href{http://www.ams.org/mathscinet-getitem?mr=#1}{#2}
}
\providecommand{\href}[2]{#2}

\end{document}